\documentclass[11pt]{article}
\usepackage{amsfonts}
\usepackage{amsmath,amssymb}
\usepackage[dvips]{graphicx}
\usepackage{algorithm}
\usepackage{algorithmicx}
\usepackage{algpseudocode}
\usepackage{yhmath}
\usepackage{qtree}
\usepackage{xcolor}
\usepackage{epsfig}
\usepackage{lineno,hyperref}
\usepackage{mathtools}

\usepackage{amsthm}

\usepackage[T1]{fontenc}
\usepackage[utf8]{inputenc}
\usepackage{booktabs}
\usepackage{authblk}
\usepackage[toc]{appendix}
\numberwithin{equation}{section}
\usepackage{verbatim}\usepackage{graphicx}
\usepackage{geometry}

\usepackage{amsmath}
\allowdisplaybreaks[1]

\newcommand{\be}{\begin{equation}}
\newcommand{\ee}{\end{equation}}
\newcommand{\ba}{\begin{array}}
\newcommand{\ea}{\end{array}}
\newcommand{\bea}{\begin{eqnarray*}}
\newcommand{\eea}{\end{eqnarray*}}
\newcommand{\bean}{\begin{eqnarray}}
\newcommand{\eean}{\end{eqnarray}}
\newcommand{\Del}{\nabla}

\newtheorem{theorem}{Theorem}[section]

\newcommand{\lc}{\mathrel{\raise2pt\hbox{${\mathop<\limits_{\raise1pt\hbox{\mbox{$\sim$}}}}$}}}
\newcommand{\gc}{\mathrel{\raise2pt\hbox{${\mathop>\limits_{\raise1pt\hbox{\mbox{$\sim$}}}}$}}}
\newcommand{\ec}{\mathrel{\raise1pt\hbox{${\mathop=\limits_{\raise2pt\hbox{\mbox{$\sim$}}}}$}}}

\begin{document}

\title{Revisit of Semi-Implicit Schemes for Phase-Field Equations}

\author[1,2]{Tao Tang}
\affil[1]{\small Division of Science and Technology, BNU-HKBU United International College, Zhuhai, Guangdong, China }
\affil[2]{\small SUSTech International Center for Mathematics, Southern University of Science and Technology, Shenzhen, China (\href{mailto:tangt@sustech.edu.cn}{tangt@sustech.edu.cn})}
\maketitle

\begin{abstract}
It is a very common practice to use semi-implicit schemes in various computations, which treat selected linear terms implicitly and the nonlinear terms explicitly. For phase-field equations, the principal elliptic operator is treated implicitly to reduce the associated stability constraints while the nonlinear terms are still treated explicitly to avoid the expensive process of solving nonlinear equations at each time step. However, very few recent numerical analysis is relevant to semi-implicit schemes, while "stabilized" schemes have become very popular. In this work, we will consider semi-implicit schemes for the Allen-Cahn equation with {\em general potential} function. It will be demonstrated that the maximum principle is valid and the energy stability also holds for the numerical solutions. This paper extends the result of Tang \& Yang (J. Comput. Math., 34(5):471--481, 2016)  which studies the semi-implicit scheme for the Allen-Cahn equation with {\em polynomial potentials}.
\end{abstract}

{\bf Keywords.} semi-implicit, phased-field equation, energy dissipation, maximum principle

{\bf AMS subject classifications: }
65M06, 65M12

{\bf Contribute to the special issue devoted to Professor Weiyi Su}

\section{Introduction}

There has been tremendous interests in developing energy-dispassion numerical methods for phase-field models starting from earlier numerical works \cite{chenshen98,eyre1998unconditionally,xu2006stability}. To make sure a numerical scheme satisfies the nonlinear energy stability, there are basically three class of approaches. For ease of exposition, we consider the simplest phase-field model, i.e., the Allen-Cahn equation with initial condition:
\begin{eqnarray}
  & \frac{\partial \phi}{\partial t} = \varepsilon^2 \Delta \phi - f(\phi), \quad &\mbox{in } \Omega\times (0,T], \label{AC0} \\
  & \phi(x, 0) = \phi_0(x), \quad & \mbox{in } \Omega, \label{AC0a}
\end{eqnarray}
where $\varepsilon>0$ is the interface width parameter, $f(\phi) = F'(\phi)$,  where $F$ is a smooth function.
The corresponding energy is defined as
\begin{equation}
  \label{energy}
  E(\phi)
  :=
  \int_{\Omega} \Bigl(\frac{\varepsilon^2}2 |\nabla \phi|^2 +  F(\phi)\Bigr)  dx.
\end{equation}
Energy stability means that
\be
E(\phi(\cdot, t)) \le E(\phi(\cdot, s)), \quad \forall t> s.
\ee
The first class of energy stable scheme is the Eyre's convex splitting method \cite{eyre1998unconditionally}, which yields a {\em nonlinear} semi-implicit scheme:
\be\label{eyre}
\frac{\phi^{n+1}- \phi^n}{\Delta t}
 =  \varepsilon^2\Delta \phi^{n+1} - T_1(\phi^{n+1}) - T_2(\phi^{n}),
\ee
where $T_1$ and $T_2$ are some convex functionals satisfying $T_1+T_2= f$. This is also referred as partially implicit scheme for phase-field modeling by \cite{xu2019stability}.

The second class of energy stable scheme is to add some extra terms so that the resulting scheme satisfies energy non-increasing property; these schemes are called "stabilized" approach. There have been quite large size of papers in this direction in the past 15 years, see the review articles \cite{shen2019new,tangICM2018} and references therein.

The third class is the direct fully implicit scheme, see, e.g., \cite{Du1991,LTZ20,xu2019stability}. In particular, it is demonstrated in \cite{xu2019stability} that a first-order fully implicit scheme for the Allen-Cahn model can be devised so that the maximum principle is valid on the discrete level and, furthermore, the linearized discretized system can be effectively preconditioned using discrete Poisson operators.

It is noted by  \cite{eyre1998unconditionally} that an unconditionally energy-stable scheme, such that backward Euler scheme, is not necessarily better than a conditionally energy stable scheme when the time step size is not small enough. In other words, if larger time steps are needed then the first and second classes schemes are useful. However, it is argued in \cite{xu2019stability} that most implicit schemes are energy-stable if the time-step size is sufficiently small. Moreover, it is noted that a convex splitting scheme or a "stabilized" approach can be equivalent to some fully implicit scheme with a different time scaling and thus it may lack numerical accuracy.

It is obvious that the partially implicit scheme and the fully implicit scheme all require some iteration techniques, and are less effective than the explicit scheme or semi-implicit scheme.
In this work, we wish to study the semi-implicit energy-stable scheme for the Allen-Cahn equation:
\be\label{semi1}
\frac{\phi_j^{n+1}- \phi_j^n}{\Delta t}
=  \varepsilon^2\Del_+ \Del_- \phi^{n+1}_j -  f(\phi^n_j), \quad 1\le j \le J,
\ee
where $\Delta t$ and $\Delta x$ are time step and mesh size in space, respectively, and
\[
\Del _+ u_j = \frac{u_{j+1}-u_j}{\Delta x}, \quad
\Del _- u_j = \frac{u_{j}-u_{j-1}}{\Delta x}.
\]
It will be demonstrated that the maximum principle is valid for \eqref{semi1}, and the energy stability also holds for the numerical solutions.

The paper is organized as follows. In Section \ref{sect2}, we will prove the maximum principe and the $L^1$-stability for the numerical solutions of semi-implicit scheme \eqref{semi1}.  The energy stability will be established in Section \ref{sect3}.  Some possible extensions will be discussed in the final section.

\section{Maximum principle}
\label{sect2}

It is known that a maximum principle is satisfied for the Allen-Cahn equation \eqref{AC0}, see, e.g.,  \cite{Trudinger2015,tang2016implicit}. Below we will provide a discrete counterpart using a
{\em monotone scheme} arguments.

\begin{theorem} \label{thm1}
Consider the semi-implicit scheme \eqref{semi1} with periodic boundary conditions. Assume $\gamma_+ >\gamma_-$ are two constants.
\begin{itemize}
\item
If the nontrival function $f\in C^1(\gamma_1, \gamma_+)$  satisfies
\be \label{2e1}
 f(\gamma_-) =f(\gamma_+)=0, \quad \gamma_- \le \phi_0 \le \gamma_+,
 \ee
 \item
 and if the time step $\Delta t$ satisfies
 \be\label{2e2}
 \Delta t \max_{\gamma_- \le u \le \gamma_+}  f'(u) \le 1,
 \ee
 \end{itemize}
 then
 \be \label{max1}
  \gamma_- \le  \phi_j  ^n \le \gamma_+, \quad \forall \;\; 1\le j \le J, \; n\ge 0.
  \ee
\end{theorem}

\begin{proof}
We use mathematical induction to prove the result. The result is obvious at $n=0$. Assume \eqref{max1} is true at level $n$. It follows from \eqref{semi1} that
\[
\phi^{n+1}_j = \phi^n_j + \lambda \left( \phi^{n+1}_{j+1} - 2\phi^{n+1}_j +\phi^{n+1}_{j-1} \right)
-\Delta t f(\phi^{n}_j) .
\]
where $\lambda := \epsilon ^2 \Delta t / \Delta x^2$. Consequently, we have
\begin{eqnarray}
(1+ 2\lambda) \phi^{n+1}_j &=& \phi^{n}_j+ \lambda \phi^{n+1}_{j+1} + \lambda \phi^{n+1}_{j-1}- \Delta t f(\phi^{n}_j) \nonumber \\
&=:& G(\phi^{n}_j, \phi^{n+1}_{j+1}, \phi^{n+1}_{j-1}) .
\end{eqnarray}
Note that
\[
\frac{\partial G}{\partial \phi^{n}_j} = 1- \Delta t f'(\phi^{n}_j) \ge 0, \quad
\frac{\partial G}{\partial \phi^{n+1}_{j+1}} = \lambda >0, \quad \frac{\partial G}{\partial \phi^{n+1}_{j-1}} =\lambda >0,
\]
where in the first inequality we have used the assumption \eqref{2e2} and the
induction assumption at $n$. The above result shows that $G$
is a monotone scheme, which yields
\begin{eqnarray}
&& (1+ 2\lambda) \phi^{n+1}_j =G(\phi^{n}_j, \phi^{n+1}_{j+1}, \phi^{n+1}_{j-1}) \nonumber \\
&\le& G( \gamma_+, \max_j \phi_j^{n+1}, \max_j  \phi^{n+1}_j)
= \gamma_+ + 2\lambda \max_j \phi^{n+1}_j,
\end{eqnarray}
where we have used the assumption that $f(\gamma_+)=0$. As the above result is true for all
$j$, we obtain
\[
(1+ 2\lambda) \max _j \phi^{n+1}_j \le \gamma_+ + 2\lambda \max_j \phi^{n+1}_j,
\]
which gives $\max _j \phi^{n+1}_j \le \gamma_+$.
Similarly, we have
\begin{eqnarray}
&& (1+ 2\lambda) \phi^{n+1}_j =G(\phi^{n}_j, \phi^{n+1}_{j+1}, \phi^{n+1}_{j-1})\nonumber \\
&\ge& G( \gamma_-, \min_j \phi_j^{n+1}, \min_j  \phi^{n+1}_j)
= \gamma_- + 2\lambda \min_j \phi_j^{n+1},
\end{eqnarray}
where we have used the assumption that $f(\gamma_-)=0$. As the above result is true for all
$j$, we obtain $\min _j \phi^{n+1}_j \ge \gamma_-$. This completes the proof of the theorem.
\end{proof}

Before closing this section, we remark that the $L^1$-stability also holds for the Allen-Cahn equation \eqref{AC0}. More precisely, if the conditions \eqref{2e1} and \eqref{2e2} in Theorem \ref{thm1} are satisfied, and also $f(0)=0$ with $\gamma_-<0< \gamma_+$, then the numerical solutions
of \eqref{semi1} satisfy
\be \label{3e0a}
 \Vert \phi^{n+1} \Vert _1  \le e^{L\Delta t} \Vert \phi^n \Vert_1,
 \ee
 where $L=-\min_{\gamma_- \le u \le \gamma_+} f'(u)$, and 
\be \label{l1norm}
\Vert \phi^n \Vert_1 = \sum ^J_{j=1} \vert \phi_j^n \vert \Delta x.
\ee
We briefly outline the proof of \eqref{3e0a}. 
It follows from the semi-implicit scheme \eqref{semi1} that
\be \label{3e1}
(1+ 2\lambda) \phi^{n+1}_j = (1- \Delta t f'(\theta^n_j) )  \phi^{n}_j+ \lambda \phi^{n+1}_{j+1} + \lambda \phi^{n+1}_{j-1},
\ee
where we have used the assumption $f(0)=0$ and $\theta ^n_j$ is between $0$ and $\phi ^n_j$.
Furthermore, it follows from \eqref{2e2} that the coefficient of $\phi^{n}_j$ is non-negative. Consequently, we have
\eqref{3e1} that
\begin{eqnarray} \label{3e2}
(1+ 2\lambda) \vert \phi^{n+1}_j \vert &\le& (1- \Delta t f'(\theta^n_j) ) \vert \phi^{n}_j \vert+ \lambda \vert \phi^{n+1}_{j+1}\vert  + \lambda \vert \phi^{n+1}_{j-1}\vert \nonumber \\
&\le& (1+ \Delta t  L)\vert \phi^{n}_j \vert+ \lambda \vert \phi^{n+1}_{j+1}\vert  + \lambda \vert \phi^{n+1}_{j-1}\vert .
\end{eqnarray}
The above result, together with the definition \eqref{l1norm}, leads to
\be
(1+ 2\lambda) \Vert \phi^{n+1} \Vert_1 \le e^{L\Delta t} \Vert \phi^{n} \Vert_1 + \lambda \Vert \phi^{n+1}\Vert_1  + \lambda \Vert \phi^{n+1}\Vert_1 .
\ee
which leads to the desired estimate \eqref{3e0a}.

\section{Energy Stability}
\label{sect3}

A numerical correspondence of the energy definition (\ref{energy}) is given below
\be \label{dis-energy}
E_h (\phi^n) = \frac{\epsilon^2}2 \sum^J_{j=1} (\Del_+ \phi^n_j)^2 \Delta x
+ \sum^J_{j=1} F(\phi^n_j) \Delta x.
\ee

\begin{theorem} \label{thm3}
Consider the semi-implicit scheme \eqref{semi1} with periodic boundary conditions.
If the conditions \eqref{2e1} and \eqref{2e2} in Theorem \ref{thm1} are satisfied, 
then the numerical solutions of \eqref{semi1} satisfy
\be \label{energy0}
 E_h(\phi^{n+1}) \le E_h (\phi^n).
 \ee
 \end{theorem}

\begin{proof}
It follows from the definition of the discrete energy \eqref{dis-energy} that
\bean
&& E_h(\phi^{n+1}) - E_h (\phi^n) \nonumber \\
&=& \frac{\epsilon^2}2 \sum^J_{j=1} (\Del_+ \phi^{n+1}_j+ \Del_+ \phi^n_j)
 (\Del_+ \phi^{n+1}_j- \Del_+ \phi^n_j)\Delta x
+ \sum^J_{j=1} \left( F(\phi^{n+1}_j)-F(\phi^n_j)\right) \Delta x \nonumber \\
&=:& \epsilon^2 I_1 + I_2. \label{i12}
\eean
We now estimate $I_1$ and $I_2$. Note that
\bean
I_1 &=& \frac 12 \sum^J_{j=1} \Del_+ (\phi^{n+1}_j+\phi^n_j)
 \Del_+ (\phi^{n+1}_j- \phi^n_j)\Delta x \nonumber \\
 &=& \sum^J_{j=1} \Del_+ \phi^{n+1}_j \Del_+ (\phi^{n+1}_j- \phi^n_j)\Delta x - \frac 12 \sum^J_{j=1} \left( \Del_+ (\phi^{n+1}_j- \phi^n_j)\right) ^2 \Delta x \nonumber \\
 &\le& \sum^J_{j=1} \Del_+ \phi^{n+1}_j \Del_+ (\phi^{n+1}_j- \phi^n_j)\Delta x\nonumber \\
&=& - \sum^J_{j=1} \Del_+ \Del_- \phi^{n+1}_j (\phi^{n+1}_j- \phi^n_j)\Delta x
\label{i1}
\eean
where in the last step we have used the discrete integration by parts. Furthermore, using Taylor expansion gives
\be
F(\phi^{n+1}_j)= F(\phi^n_j)+ f(\phi^n_j) (\phi^{n+1}_j-\phi^n_j) + \frac 12 f'(\theta^n_j) (\phi^{n+1}_j-\phi^n_j)^2,
\ee
where $\gamma_- \le  \theta_j  ^n \le \gamma_+$, which yields
\be \label{i2}
I_2 = \sum^J_{j=1} \Big( f(\phi^n_j) (\phi^{n+1}_j-\phi^n_j) + \frac 12 f'(\theta^n_j) (\phi^{n+1}_j-\phi^n_j)^2\Big) \Delta x.
\ee
It follows from the \eqref{i12}, \eqref{i1} and \eqref{i2} that
\bean
&& E_h(\phi^{n+1}) - E_h (\phi^n)  \nonumber \\
&\le& \sum^J_{j=1} \Big[ -\epsilon ^2 \Del_+ \Del_ - \phi^{n+1}_j (\phi^{n+1}_j- \phi^n_j)+ f(\phi^n_j) (\phi^{n+1}_j-\phi^n_j) \Big] \Delta x \nonumber \\
&& \quad + \sum^J_{j=1} \frac 12 f'(\theta^n_j) (\phi^{n+1}_j-\phi^n_j)^2 \Delta x.
\eean
It follows from the semi-implicit scheme \eqref{semi1} that
\[
\frac{(\phi_j^{n+1}- \phi_j^n)^2}{\Delta t}
=  \varepsilon^2\Del_+ \Del_- \phi^{n+1}_j(\phi_j^{n+1}- \phi_j^n) -  f(\phi^n_j)(\phi_j^{n+1}- \phi_j^n),
\]
Combining the above two results gives
\bean
&& E_h(\phi^{n+1}) - E_h (\phi^n)  \nonumber \\
&\le& \sum^J_{j=1} \frac{\Delta x} {\Delta t} \left(-1 + \frac 12 \Delta t f'(\theta^n_j) \right)
(\phi^{n+1}_j-\phi^n_j)^2 \le 0,
\eean
where in the last step we have used the assumption \eqref{2e2} and the maximum principle result \eqref{max1}.
\end{proof}

\section{Concluding Remarks}

This paper extends the result of \cite{tang2016implicit} which provides similar maximum principle and energy stability results for the polynomial double well potential $F$. In \cite{tang2016implicit}, numerical experiments also demonstrated that the semi-implicit (sometimes also called implicit-explicit) method is an effective scheme for solving the Allen-Cahn equations.

It is demonstrated in \cite{chenshen98} that semi-implicit scheme is very effective for the Ginzburg-Landau equation and the Cahn-Hilliard equation. We may conjecture that the energy stability \eqref{energy0} also holds for the semi-implicit solutions of the Cahn-Hilliard equation. However, standard energy analysis used in this work may not be sufficient to establish such a result. It is expected some deeper analysis is needed to verify the conjecture.

We point out that it is possible to extend the present framework to handle the Cahn-Hilliard equation with a logarithmic free energy. The key ingredient for establishing the energy stability with the present framework is the boundedness of numerical solutions. The solution boundedness is indeed true for the Cahn-Hilliard equation with a logarithmic free energy  \cite{Cherfils11}; however, there has very few theoretical justification of the numerical counterpart. It is worth mentioning two relevant works in this direction. One is the work of Copetti and Elliott \cite{elliott92} who analyzed the implicit Euler scheme and obtained the uniform maximum bound of the numerical solutions; another one is due to Chen et al. \cite{chenwb19} who studied the first-order and second-order partially implicit scheme and obtained the $L^\infty$-bound of the numerical solutions. Of course, more interesting and challenging issue is to analyze the energy stability of semi-implicit schemes for the Cahn-Hilliard equation with a logarithmic free energy.


\end{document}